\documentclass[11pt]{article} 

\usepackage{geometry} 
\usepackage{booktabs} 
\usepackage{array} 
\usepackage{paralist} 
\usepackage{verbatim} 
\usepackage{subfig} 
\usepackage{amssymb,amsmath, amsthm} 
\usepackage{outlines}

\newcommand{\R}{\mathbb{R}}

\usepackage{fancyhdr} 
\usepackage{sectsty}
\allsectionsfont{\sffamily\mdseries\upshape} 

\usepackage[nottoc,notlof,notlot]{tocbibind} 
\usepackage[titles,subfigure]{tocloft} 

\usepackage{pdfpages}
\usepackage{graphicx}
\usepackage{subfig}

\newtheorem{theorem}{Theorem}
\newtheorem{corollary}{Corollary}[theorem]
\newtheorem{lemma}{Lemma}

\newtheorem{definition}{Definition}[section]

\title{Multiflows: A New Technique for Filippov Systems and Differential Inclusions}
\author{Cameron Thieme}
\date{\today}

\begin{document}
\sloppy
\maketitle

\begin{abstract}
    The \textit{flow} is very useful in studying dynamical systems.  However, many modern systems--notably differential inclusions--do not have unique solutions, and therefore cannot be described by flows.  Richard McGehee has proposed an object, the \textit{multiflow}, in order to use the topological techniques developed for flows in this setting.  In this paper we will introduce multiflows and prove that under basic conditions, differential inclusions give rise to multiflows.  This paper will also outline the most prominent example of differential inclusions, Filippov systems, as motivation.  In addition, several  results on differential inclusions, such as the existence theorem, are reproven here in order to create a self-contained work.  
\end{abstract}

\section{Introduction} $\,$\par

In dynamical systems we often consider differential equations which take the form
$$\dot{x} = f (x)$$ where $x$ is a point on an open set $U \subset \R^n$ and $f : U \mapsto \R^n$ is a smooth function.  The meaning of smooth varies in literature, but generally $f$ is assumed at least to be
Lipschitz continuous in order to guarantee a unique solution to the differential equation. By a
solution of this equation we mean any differentiable function $x(t)$ such that
$$\frac{d}{dt}(x(t)) = f (x(t))$$ for all $t$ in some open interval $I$. Since we are considering autonomous systems, we assume that $I$ contains the time $t = 0$.\par 
For the equation described above, there is actually a whole family of solutions, one for each initial condition. In order to consider the entire family of solutions as a single object we study \textit{flows}. A flow is a continuous map $\varphi : \R \times \R^n \mapsto \R^n$ such that
\begin{enumerate}
\item $\varphi(0, x_0 ) = x_0$
\item $\varphi(s, \varphi(t, x_0 )) = \varphi(s + t, x_0 )$
\end{enumerate}\par 
This concept relates to the differential equation by letting $\varphi(t, x_0 )$ be the solution $x(t)$
with the initial condition $x(0) = x_0$.\par 
As long as $f$ is Lipschitz continuous, the equation $\dot{x} = f (x)$ generates a unique local
flow $\varphi : U \subset \R \times \R^n \mapsto \R^n$ . Such equations and flows have been extensively studied for
over a century.\par 
However, differential equations where $f$ is not Lipschitz continuous have received much less study. The reasons for this omission are essentially twofold. Firstly, from a
mathematical point of view, the Lipschitz assumption is extremely valuable. Without it, solutions are not guaranteed to be unique up to the initial condition, and hence
we lose determinism (and therefore the concept of a flow). Second, this mathematical simplification was historically justified by the applications that scientists studied.
Most vector fields of interest were Lipschitz continuous, and hence the more complicated study of equations lacking this smoothness seemed unnecessary.\par 
However, this second point is becoming less true in the modern world. There are now many models where the underlying differential equations are not Lipschitz continuous, or even continuous.  This set includes models of friction, where an object can reach a restpoint in finite time, and models involving mechanical switching, where a solution evolves according to one vectorfield till it reaches a certain point and then switches to another \cite{bernardo}. Low dimensional climate models also frequently exhibit non-smooth behaviour \cite{welander}.\par 
Thus it has become important to begin to study differential equations with discontinuous right-hand sides. This new area of study presents many unique challenges. Notably, the current definition of a solution does not work for these situations; in fact, even the definition of a differential equation must be altered. These considerations have led to the formulation of \textit{differential inclusions}.

\begin{definition} A \textbf{differential inclusion} is a generalization of the concept of a differential equation. It takes the form
$$\dot{x} \in F (x)$$
where $F$ is a set-valued map\footnote{For clarity, all set valued functions in this paper will be capitalized, like $F (x)$, while single valued maps will be lower-case, like $f (x)$
}.\par 
A \textbf{solution of the differential inclusion} is an absolutely continuous function $x(t)$ defined on some interval $I \in \R$ such that
$$\frac{d}{dt} x(t) \in F (x(t))$$ almost everywhere in $I$.
\end{definition}\par 
On a compact interval $[a,b]$, an absolutely continuous function $x(t)$ may be written as 
$$x(t)= x(a)+\int_a^t\dot{x}(s)ds $$
where the derivative $\dot{x}$ is Lebesgue integrable and exists almost everywhere.  The motivation for considering these almost everywhere differentiable functions as the solutions to differential inclusions, rather than $C^1$ functions as we do for differential equations, will become more clear when we consider one of the most prominent examples of differential inclusions, Filippov systems.\par 
These differential inclusions have been studied in recent years, most notably by A.F. Filippov \cite{filippov}. However, to date, no suitable generalization of the concept of a flow has been found for these systems.  Richard McGehee has defined an object, called a \textit{multiflow}, which remedies that situation; that is, the multiflow is a generalization of a flow suitable to a broad class of differential inclusions.\par 
This paper is split into three main sections.  In the first section, we introduce the basic conditions on a set-valued map $F$ of a differential inclusion $\dot{x}\in F(x)$ and explain how differential equations with piecewise-continuous righthand sides can be reformulated as such an inclusion.  This section also references some common scientific applications of these systems.  In the following section we examine some elementary results on differential inclusions, like solution existence.  The theorems in this section are all found in Filippov's book \cite{filippov}, and are presented here in order to create a self-contained exposition.  In the final section we define the concept of a multiflow and prove the main result, that differential inclusions give rise to multiflows.

\section{Differential Inclusions}

\subsection{Upper Semicontinuous Differential Inclusions}

In order to analyze differential inclusions $\dot{x}\in F(x)$, we must first put some conditions on the set-valued map $F$.  These conditions are extremely general, and most differential inclusions of scientific interest will meet the requirements.  The most notable condition is upper-semicontinuity of the correspondence $F$.

\begin{definition}
Let $X$ and $Y$ be topological spaces.  A set-valued function $F:X\to Y$ is said to be \textbf{upper semicontinuous at the point $x$} if for any neighbourhood $V$ of $F(x)$, there exists some neighbourhood $U$ of $x$ such that  $F(U)\subset V$.\par 
Then $F$ is said to be \textbf{upper semicontinuous} if it is upper semicontinuous at each $x\in X$.
\end{definition}

This definition of an upper semicontinuous set-valued function is the most general one, and applies to set-valued functions between any topological spaces.  But in this paper we will only consider set-valued functions in $\R^n$, and so it is often more convenient for us to consider $\epsilon$ and $\delta$ neighbourhoods rather than arbitrary neighbourhoods.  We can then rewrite the definition of upper semicontinuity in these terms:

\begin{definition}
A set-valued function $F:G\subset\R^n\to\R^n$ is said to be \textbf{upper semicontinuous at the point $x$} if for any $\epsilon>0$ there exists some $\delta>0$ such that $F(B_\delta(x))$ is a subset of an open $\epsilon$-neighbourhood of $F(x)$.\par 
The correspondence $F$ is said to be \textbf{upper semicontinuous} if it is upper semicontinuous at each $x\in G$.
\end{definition}

The above definition is stated in terms of open $\epsilon$ and $\delta$ neighbourhoods.  However, throughout this paper we will work almost exclusively with closed sets.  Therefore, in order to bring this definition in line with the rest of the work we will use closed $\epsilon$ and $\delta$ neighbourhoods instead of open ones.  It is a simple exercise to show that the definition is equivalent whether we use open or closed neighbourhoods.  The closed $\delta$-neighborhood of a set $S$ is denoted $\overline{N_\delta(S)}$.\par 


We now give the final, most succinct definition of an upper semicontinuous set-valued function:

\begin{definition}
A set-valued function $F:G\subset\R^n\to\R^n$ is said to be \textbf{upper semicontinuous at the point $x$} if for any $\epsilon>0$ there exists some $\delta>0$ such that $F(\overline{N_\delta(x)})\subset \overline{N_\epsilon(F(x))}$.\par 
The correspondence $F$ is said to be \textbf{upper semicontinuous} if it is upper semicontinuous at each $x\in G$.
\end{definition}

In addition to upper semicontinuity, Filippov also introduces what he calls the \textit{basic conditions} for a set-valued $F(x)$ on a domain $G$.  These conditions are necessary in order to get an existence result for differential inclusions, and so we will also assume these same conditions in this paper.

\begin{definition}  Let the set-valued function $F:G\subset\R^n\to\R^n$ be upper semicontinuous in $x$.  Also, for all $x_0\in G$, assume that the set $F(x_0)$ is
\begin{itemize}
 \item non-empty
 \item bounded
 \item closed
 \item convex  
 \end{itemize}
 Then $F$ is said to satisfy the \textbf{basic conditions}.
 \end{definition}\par 

These conditions are very general, applying to a wide variety of dynamical systems.  Note that any continuous single-valued function $f(x)$ trivially satisfies these conditions, and so any results about differential inclusions $\dot{x}\in F(x)$ also apply to more typical differential equations $\dot{x}=f(x)$.

\subsection{Piecewise-Continuous Differential Equations}$\,$\par
The study of upper semicontinous differential inclusions is motivated primarily by differential equations with nonsmooth righthand side.  The meaning of nonsmooth varies in the literature, and a good deal of study has been devoted to differential equations with varying degrees of differentiability \cite{bernardo}.  Here, however, we will examine the setting of piecewise continuous differential equations, and all results will apply to the broad class of systems which may have discontinuities in the underlying vector fields.  In this section, we will rigorously define these systems on an open domain $G$; the definitions and formulations presented here draw heavily from the work of A.F. Filippov \cite{filippov}.\par 
Without loss of generality, we may assume that $G$ is connected because disconnected portions may simply be examined separately.  The set $G$ is divided into open, disjoint regions $G_i$, along with their boundary points.  We will assume that the set of all boundary points of all $G_i$ is measure zero in $G$, and denote it by $\Sigma$.  We will refer to $\Sigma$ as the \textit{splitting boundary}.  For analytical reasons, we also impose the additional condition that any compact subset of $G$ contains only finitely many $G_i$; this assumption is very useful mathematically, and does not impose a burden from a modelling standpoint.  For the rest of this paper, we will call any connected, open domains partitioned in this way \textit{Filippov domains}.  This language is not standard in the literature, but it is useful to reference for our purposes.  \par 
We now consider a set of differential equations defined on the Filippov domain:
$$\dot{x} = f_i (x),\hspace{1cm} x \in G_i \subset G$$\par
Each $f_i$ is required only to be continuous, framing the system as a piecewise continuous one. 
We also assume that $f_i$ are continuous up to the boundary of $G_i$ so that $f_i (x)$ evaluates to a finite vector for each $i$ and for all $x \in \Sigma$.  In other words, each $f_i$ is defined and continuous on the closure of $G_i$.\par  
Of course, as written, this system is incomplete; there is no information about the vectorfield along $\Sigma$, and so there is no way to continue a solution which reaches the boundary of any $G_i$.  This issue brings us to the concepts of differential inclusion and the Filippov convex combination method.\par 
At each $x\in\Sigma$, there are multiple vector fields $f_i(x)$ that are defined. Because of that fact, it makes sense to introduce a differential inclusion $\dot{x}\in F(x)$.  In the Filippov convex combination method we define our set-valued vector field $F (x)$ to be the single-valued functions $f_i(x)$ for all $x$ in any of the open regions $G_i$.  For $x\in\Sigma$, however, we take $F(x)$ to be the set-valued convex hull of all vectors $f_i(x)$ such that $x$ is a boundary point of $G_i$. We will collect all of this information in the following definition.\par 
\begin{definition}
Let $G$ be a Filippov domain such that each $G_i$ is associated with a function $f_i$ that is continuous in the closure of $G_i$.
Define a set-valued function $F$ in the following way.\\
For $x\in G_i$, let $$F(x)=\{f_i(x)\}$$ 
For $x\in\Sigma$, let $F(x)$ be the convex hull of all vectors $f_i(x)$ such that $x$ is a boundary point of $G_i$.\\  
Then the differential inclusion $\dot{x}\in F(x)$ is a \textbf{Filippov System}.
\end{definition}
Note that any differential equation $\dot{x}=f(x)$ where $f$ is continuous will trivially fit into this framework.  Classical systems, with Lipschitz continuous differential equations, may then be viewed as special cases of Filippov systems.\par 
Here it is worth returning to the definition of a solution to a differential inclusion, as these Filippov systems show us why we want to only demand that our solutions be differentiable almost everywhere.  When a solution approaches the splitting boundary $\Sigma$, its derivative limits to a certain value.  However, on the other side of that boundary, the solution need not continue at the same velocity since the defining vectorfields $f_i$ and $f_j$ need not have any relationship to one another.  Hence we typically expect a loss of differentiability when solutions reach $\Sigma$. \par
Now that we have defined a Filippov system $\dot{x}\in F(x)$, we should check that it does, in fact, meet the basic conditions.
\begin{lemma} 
The set valued function $F$ defined by a Filippov system is upper semicontinuous.
\end{lemma} 

\begin{proof}
It is clear that $F$ is upper-semicontinuous at $x\in G_i$ for any $i$ since $F(x)$ is defined by the single-valued continuous function $f_i(x)$ at such points.  Thus, it remains only to show that $F$ is upper-semicontinuous at $x\in\Sigma$.\par Fixing such an $x$, the set $F(x)$ is the convex hull of a finite number of vectors $\{f_i(x)\}_{i=1}^p$.  Consider an arbitrary closed $\epsilon$-neighbourhood of $F(x)$, $\overline{N_\epsilon(F(x))}$.  For each $i$, there is some $\delta_i$ such that $y\in \overline{G_i}$ and $|x-y|\leq\delta_i$ implies that $|f_i(x)-f_i(y)|\leq\epsilon$, which we can also write as $f_i(y)\in \overline{N_\epsilon(f_i(x))}$.  Let $\delta=\min_{1\leq i\leq p}\{\delta_i\}$.  Then for $y\in \overline{N_\delta(x)}$, $F(y)$ is either a single valued function $f_i(y)$ (for $y\in G_i$) or the the convex hull of a finite set of vectors $\{f_i(y)\}_{i=1}^q$ (for $y\in\Sigma$) all satisfying the relationship $f_i(y)\in \overline{N_\epsilon(f_i(x))}\subset \overline{N_\epsilon(F(x))}$.   \par 
This fact implies that for $y\in \overline{N_\delta(x)}$, $F(y)\subset \overline{N_\epsilon(F(x))}$; to verify this statement, consider an arbitrary vector $f_y\in F(y)$.  We can write this vector as $f_y=\sum_{i=1}^q \alpha_i f_i(y)$, where $1\leq q\leq p$ and $\sum_{i=1}^q \alpha_i=1$.  Now consider the vector $f_x:=\sum_{i=1}^q \alpha_i f_i(x)\in F(x)$.  
\begin{align*}
    |f_x-f_y|&=|\sum_{i=1}^q \alpha_if_i(x)-\sum_{i=1}^q\alpha_if_i(y)|\\
    &= |\sum_{i=1}^q \alpha_i(f_i(x)-f_i(y))|\\
    &\leq \sum_{i=1}^q\alpha_i|f_i(x)-f_i(y)|\\
    &\leq \sum_{i=1}^q\alpha_i\epsilon\\
    &=\epsilon
\end{align*}
Thus $F(\overline{N_\delta(x)})\subset \overline{N_\epsilon(F(x))}$, and so $F$ is upper-semicontinuous at any $x$.
\end{proof}
With this lemma, it is easy to see that Filippov systems satisfy the basic conditions of differential inclusions.  At each point, the correspondence $F$ is well-defined and non-empty.  Wherever the system is single valued, $F(x_0)$ is clearly compact and convex.  For $x_0$ where the system is not single valued, $F(x_0)$ is still clearly closed and convex by definition, and it is bounded because we assume that $f_i$ is defined on $\overline{G_i}$ for each $i$, and so $f_i(x_0)$ is always finite.  We will summarize this information into a theorem:
\begin{theorem}
The set valued function $F$ defined by a Filippov system satisfies the basic conditions.
\end{theorem}
\par 
In general it is possible to consider Filippov systems where $\Sigma$ is very complicated, and when theorems are quoted in this paper they will apply to these general systems defined above.  However, in most models, $\Sigma$ is simply a codimension-1 manifold.  There are a few notable common examples where the dimensionality of $\Sigma$ is not well defined everywhere--for instance, $\Sigma$ could be two intersecting lines--and so it would not be a manifold.  But usually, the set $\Sigma$ is a manifold.\par 
In fact, the majority of nonsmooth models in the literature have  only two distinct regions, which we will refer to as $G_-$ and $G_+$. When only two regions border $\Sigma$, as in this case, the convex hull may be written compactly as the convex combination of the two vectors $f_-(x)$ and $f_+(x)$:
$$F(x)=\{\alpha f_+(x) + (1-\alpha)f_-(x):\alpha\in[0,1]\}\hspace{1cm} x\in\Sigma$$
\begin{figure}[h]
  \centering
  \subfloat{\includegraphics[width=0.4\textwidth]{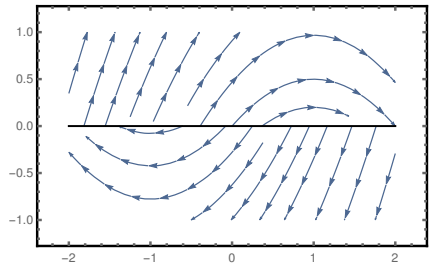}\label{fig:f1}}
  \caption{An example of a simple Filippov system in $\R^2$\cite{leifeldcode}.  The splitting boundary is the $x$-axis.  The region above the $x$-axis is $G_+$ and the region below it is $G_-$.}
\end{figure}\par 
Thus the entire Filippov system may be written as
$$\dot{x}\in F(x) =
\begin{cases}
f_- (x), & x \in G_- \\
f_+ (x), & x \in G_+ \\
\{\alpha f_+(x) + (1-\alpha)f_-(x):\alpha\in[0,1]\} & x\in\Sigma
\end{cases}$$\par 
We see that $\dot{x}$ is dependent on $\alpha$, and so solutions which reach the splitting boundary do not necessarily obey deterministic laws. This loss of determinism makes it impossible to expect Filippov systems to give rise to flows.  However, because flows are valuable tools in the study of dynamical systems we would still like some generalization of the concept; we hope that this generalization will be the obeject we call a multiflow.\par 

\begin{figure}[h]
  \centering
  \subfloat{\includegraphics[width=0.4\textwidth]{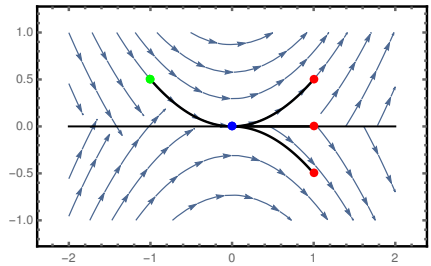}\label{fig:f1}}
  \caption{A simple Filippov System.  Solutions beginning at the green dot are no longer uniquely determined once they hit the splitting boundary at the blue dot, and may end up at any of the red dots (or infinite other locations).}
\end{figure}\par

Before introducing multiflows, however, we will present some basic results on upper semicontinuous differential inclusions.

\section{Theorems about Solutions to Differential Inclusions}$\,$\par 
Now that we have a few examples of upper semicontinuous differential inclusions, we would like to know more information about these systems and their solutions.  Fortunately, Filippov developed and compiled a good deal of the useful machinery in his seminal work \cite{filippov}.  We will restate and reprove some of the theorems that appear in that work so that this paper can be a self-contained introduction to the concept of multiflows.\par 
Probably the most important result that Filippov proves is that the basic conditions on a set-valued map $F$ guarantee the existence of solutions to the differential equation $\dot{x}\in F(x)$.  He also shows that these solutions behave something like solutions to standard differential equations.\par 
The first of these results is analogous to the fact that continuous functions are bounded on compact sets.

\begin{lemma}\label{bounded} \cite{filippov}
If $F$ satisfies the basic conditions in a closed, bounded domain $D$ then there is some $M\in\R$ such that $|F(x)|\leq M$ for all $x\in D$.
\end{lemma}

By $|F(x)|\leq M$, we mean that if $f\in F(x)$ then $|f|\leq M$.

\begin{proof}
If this result were not true then we could choose a sequence $\{x_i\}\in D$ such that $|F(x_i)|\to\infty$ monotonically as $i\to\infty$.  Since $D$ is compact, we can find a convergent subsequence $x_{i_j}\to x^*\in D$.  Since we have assumed that $F$ is bounded at each point, $|F(x^*)|<\infty$.  But by the upper semicontinuity of $F$, for any $\epsilon>0$ and correspondingly large $j$, $F(x_{i_j})\subset \overline{N_\epsilon(F(x^*))}$.  This inclusion contradicts the assumption that $|F(x_i)|\to\infty$ monotonically.
\end{proof}

The next few results build to the existence theorem for differential inclusions.  The proof of that theorem is a modification of the classic Cauchy-Peano existence proof for differential equations.  The Cauchy-Peano proof relies on the construction of approximate solutions to the differential equation $\dot{x}=f(x)$, and so it is therefore necessary to define an approximate solution to a differential inclusion $\dot{x}\in F(x)$.\par
Before defining these approximate solutions, we will introduce some notation: if $A$ is any subset of $\R^n$, then $co(A)$ denotes the smallest convex set containing $A$.

\begin{definition}
A \textbf{$\delta$-solution} of the differential inclusion $\dot{x}\in F(x)$ is an absolutely continuous function $y(t)$ that almost everywhere satisfies the differential inclusion 
$$\dot{y}(t)\in F_\delta(y(t))$$
where $F_\delta(y):= \overline{N_\delta(co(F(\overline{N_\delta(y)})))}$
\end{definition}

One of the key ideas of the Cauchy-Peano existence proof is that a sequence of increasingly accurate approximate solutions to the differential equation converges to an exact solution.  This step is very different in the case of differential inclusions, and so we will present the analogous result in the next two lemmas.  

\begin{lemma}\label{abscon} \cite{filippov} 
Let $\{x_k:[a,b]\to\R^n\}_{k=1}^\infty$ be a sequence of absolutely continuous functions that limit to a function $x(t)$, and assume that $\dot{x_k}(t)\in D$ almost everywhere, where $D\subset\R^n$ is a compact, convex set.  Then $x(t)$ is absolutely continuous and $\dot{x}(t)\in D$ wherever it is defined, namely, almost everywhere on $[a,b]$.  
\end{lemma}

\begin{proof}
Since $D$ is bounded, there is some $m>0$ such that $|\dot{x_k}(t)|\leq m$ for all $k$ and $t\in[a,b]$. Then for $t_1,t_2\in[a,b]$  we have:
\begin{align*}
    |x(t_1)-x(t_2)|&=\lim_{k\to\infty} |x_k(t_1)-x_k(t_2)|\\
    &=\lim_{k\to\infty} |\int_{t_2}^{t_1}\dot{x_k}(t)\,dt|\\
    &\leq \lim_{k\to\infty} \int_{t_2}^{t_1}m\,dt\\
    &= m|t_1-t_2|
\end{align*}
Thus $x$ is Lipschitz continuous, and hence absolutely continuous.\par 
To see that $\dot{x}(t)\in D$ wherever it is defined, arbitrarily fix $t\in(a,b)$ and take $h$ small enough that $[t-h,t+h]\subset(a,b)$.  We claim that 
$$q_k^h:=\frac{x_k(t+h)-x_k(t)}{h}=\int_t^{t+h}\frac{\dot{x_k}(t)}{h}\,dt\in D$$\par 
In order to prove this claim, we consider the Riemann definition of the above integral. Note that since the $x_k$ are absolutely continuous functions on the real line, the Riemann and Lebesgue definitions of the integral are equivalent, and so we consider the Riemann sum for simplicity.  The $q_k^h$ are supremums (or infimums) of integral sums of the form
$$\sum\frac{\Delta_i \dot{x_k}(t_i)}{h},\hspace{1cm} \sum\frac{\Delta_i }{h}=1$$\par 
This presentation shows that because of the averaging $\frac{1}{h}$ factor, the integral sums are convex combinations of points $\dot{x_k}(t_i)\in D$, and hence belong to the convex set $D$.  Since $D$ is compact, the supremum (or infimum) over the set of integral sums also belongs to $D$, and so the claim is verified.\par 
Since $D$ is compact, $$\lim_{k\to\infty}q_k^h=\frac{x(t+h)-x(t)}{h}\in D$$
Note that the above statement remains true for arbitrarily small $h$.  Then again using the compactness of $D$, this statement implies that $$\dot{x}(t)=\lim_{h\to 0}\frac{x(t+h)-x(t)}{h}=\lim_{h\to 0}q_k^h\in D$$ whenever that limit exists.  Since $x(t)$ is absolutely continuous, the limit must exist almost everywhere on the interval $(a,b)$.

\end{proof}

\begin{lemma} \label{unisol} \cite{filippov} 
Let $F(x)$ satisfy the basic conditions in a domain $G$ and let $\delta_k\to 0$ as $k\to\infty$.  Then the limit $x(t)$ of a uniformly convergent sequence $\{x_k:[a,b]\to G\}$ of $\delta_k$-solutions  to the differential inclusion $\dot{x}\in F(x)$ is a solution to that inclusion (as long as $x(t)\in G$).
\end{lemma}

\begin{proof}
Choose an arbitrary $t_0\in(a,b)$ and $\epsilon>0$.  We will show that in a neighbourhood of $t_0$, $\dot{x_k}(t)\in \overline{N_{2\epsilon}(F(x(t_0)))}$.  By lemma \ref{abscon}, this relationship implies that in that neighbourhood, $x(t)$ is absolutely continuous and $\dot{x}(t)\in \overline{N_{2\epsilon}(F(x(t_0)))}$ wherever that derivative exists (almost everywhere in the neighbourhood). Since our choices of $t_o,\epsilon$ were arbitrary, we see a few things.  First, we see that $x(t)$ is absolutely continuous on the whole of the interval $[a,b]$ since it is absolutely continuous in a neighbourhood of each $t_0$ and $[a,b]$ is compact.  Second, it shows that $\dot{x}(t_0)\in F(x(t_0))$ for all $t_0$ where $x$ is differentiable (again, almost everywhere on the interval) since the choice of $\epsilon$ was arbitrary and $t_0$ clearly belongs to any neighbourhood of itself. Thus, in order to prove this lemma, we only need to prove the claim.\par 
During this proof, bear in mind that we only consider $t\in[a,b]$, and so for $t_0=a$ or $t_0=b$ the neighbourhoods we will find are one-sided.  Now, let $x_0:=x(t_0)$.  By the upper-semicontinuity of $F$, there exists some $\eta>0$ such that $|y-x_0|<3\eta$ implies that $F(y)\subset \overline{N_\epsilon(F(x_0))}$.  Since $\delta_k\to 0$ and the $x_k(t)$ converge uniformly to $x(t)$, there is also some $k_0$ such that $k>k_0$ implies that $\delta_k<\min(\eta,\epsilon)$ and $|x_k(t)-x(t)|<\eta$ for all $t\in[a,b]$.  Additionally, by the continuity of $x(t)$ (clear since $x_k\to x$ uniformly), there is some $\gamma\in(0,\eta)$ such that $|t-t_0|<\gamma$ implies that $|x(t)-x(t_0)|<\eta$.  
For such $t,k,\eta$ and $\gamma$, the following are true:
\begin{enumerate}
\item $\overline{N_{\delta_k}(t)}\subset \overline{N_{2\eta}(t_0)}$:  This relationship is clear from the choice of $t$.
\item $\overline{N_{\delta_k}(x_k(t))}\subset \overline{N_{3\eta}(x_0)}$:  This fact follows from our assumption that ${\delta_k}<\eta$ and the inequalities
$$|x_k(t)-x(t_0)|\leq|x_k(t)-x(t)|+|x(t)-x(t_0)|\leq\eta+\eta$$
\item $F(\overline{N_{\delta_k}(x_k(t))})\subset\overline{N_\epsilon(F(x_0))}$: This final relationship follows from the prior one and our choice of $\eta$.  
\end{enumerate}\par 
From this third insight we see that
\begin{align*}
\dot{x_k}(t)&\in \overline{N_{\delta_k}(co(F(\overline{N_{\delta_k}(x_k(t))})))}\\
&\subset \overline{N_{\delta_k}(co(\overline{N_\epsilon(F(x_0))}))}\\
&\subset \overline{N_{2\epsilon}(F(x_0))}
\end{align*}
The final inclusion follows from the condition that ${\delta_k}<\epsilon$ and the fact that $F(x_0)$, and hence $\overline{N_\epsilon(F(x_0))}$, are already convex.  By lemma \ref{abscon}, it follows that $\dot{x}(t_0)\in \overline{N_{2\epsilon}(F(x_0))}$, completing the proof.  

\end{proof}
The preceding lemma also gives us the following corollary, which is used in showing that basic differential inclusions give rise to multiflows.
\begin{corollary}
If $F(x)$ satisfies the basic conditions in $G$, then the limit of a uniformly convergent sequence of solutions to the differential inclusion $\dot{x}\in F(x)$ is also a solution. 
\end{corollary}
Now that we have that lemma, we are in a position to prove the main existence result.  The proof of that theorem is very similar to the proof of the classic Cauchy-Peano existence theorem, using a sequence of Euler broken line approximations that limit to the desired solution.  One small alteration that must be made is that when we iteratively define the Euler broken lines at a point $x$, we choose any arbitrary vector in $F(x)$ since we do not have a unique choice $f(x)$.  The fact that these approximate solutions of the differential inclusion converge to an exact solution follows from lemma \ref{unisol}.\par
\begin{theorem} \label{exist} \cite{filippov}
 Let $F$ satisfy the basic conditions in an open domain $G\subset\R^n$.  Then for any $x_0\in G$, there exists a solution of the differential inclusion
$$\dot{x}\in F(x),\hspace{1cm} x(0)=x_0$$
on some interval $[-c_-,c_+]$, where $c_-,c_+>0$.
\end{theorem}  

\begin{proof}
Without loss of generality, we will demonstrate solution existence on a closed positive interval $[0,c]$.  The existence proof in backwards time is symmetric.\par
Since $G$ is open, we may choose $r$ small enough that the closed ball $\overline{B_r(x_0)}$ is contained in $G$.  Let us denote this ball $Z$.  Next, let $m:=\sup_Z |F(x)|$.  By lemma \ref{bounded}, $m<\infty$.  The length of our interval is $c:=\frac{r}{m}$.  We are now ready to begin to define the sequence of Euler broken lines.\par 
For $k=1,2,\cdots$, define a step size $h_k:=\frac{c}{k}$; clearly, $h_k\to 0$ as $k\to\infty$.  We partition the interval $[0,c]$ into $k$ subintervals.  Let $t_k^i:=ih_k$ for $i=0,1,\cdots,k$.  Note that the superscript here is an index rather than an exponential.  We will define a family of continuous functions $x_k:[0,c]\to Z$ that are linear on the intervals $[t_k^i,t_k^{i+1}]$.\par 
We initiate an iterative process by declaring that $x_k(0)=x_0$.  In order to define $x_k(t)$ for $t\in(t_k^i,t_k^{i+1}]$, first choose any vector $v_k^i\in F(x_k(t_k^i))$.  Again, the superscript here denotes an index.  Then for $t\in(t_k^i,t_k^{i+1}]$, 
$$x_k(t):=x_k(t_k^i)+(t-t_k^i)v_k^i$$
The functions $x_k(t)$ are absolutely continuous since they are continuous and piecewise linear.  Additionally, if we define $\delta_k:=h_k$, then 
$$\dot{x_k}(t)= v_k^i\in F(x_k(t_k^i))\subset F(x_k(\overline{N_{\delta_k}(t)}))\subset F_{\delta_k}(x_k(t)) $$
and so the $x_k$ are $\delta_k$-solutions to the differential inclusion $\dot{x}\in F(x)$.\par 
We also see that $x_k(t)\in Z$ for $t\in[0,c]$ because we make at most $k$ steps of length $h_k=\frac{r}{mk}$ and the maximum velocity is $m$.  More formally, for $t\in(t_k^{l},t_k^{l+1}]$ ($0\leq l< k$), we have the following:
\begin{align*}
|x_k(t)-x_0|&=|\int_0^t \dot{x_k}(s)\,ds|\\
&=|\int_{t_k^l}^{t}\dot{x_k}(s)\,ds+\sum_{i=0}^{l-1}\int_{t_k^i}^{t_k^{i+1}}\dot{x_k}(s)\,ds|\\
&\leq \int_{t_k^l}^t|\dot{x_k}(s)|\,ds+\sum_{i=0}^{l-1}\int_{t_k^i}^{t_k^{i+1}}|\dot{x_k}(s)|\,ds\\
&= \int_{t_k^l}^t|v_k^l|\,ds+\sum_{i=0}^{l-1}\int_{t_k^i}^{t_k^{i+1}}|v_k^i|\,ds\\
&\leq \sum_{i=0}^k\int_{t_k^i}^{t_k^{i+1}}m\,ds\\
&= \sum_{i=0}^k(h_km)\\
&= k(\frac{r}{mk})m\\
&=r
\end{align*} $\,$\par 
Then since the family of functions $\{x_k\}_{k=1}^\infty$ is uniformly bounded (contained in $Z$) and equicontinuous ($|\dot{x_k}(t)|\leq m$), by the Arzela-Ascoli theorem we can choose a uniformly convergent subsequence with a limit $x(t)$.  Since $Z$ is compact, $x(t)\in Z$ for $t\in[0,c]$, and so by lemma \ref{unisol}, the function $x:[0,c]\to G$ is a solution of the inclusion $\dot{x}\in F(x)$.
\end{proof}

The basic existence result is very important for the study of differential inclusions.  Unfortunately, there is no general uniqueness result for basic differential inclusions; many of these systems do, in fact, have multiple solutions for a given initial condition.  However, the solutions of these differential inclusions do behave in other ways that are reminiscent of solutions to standard differential equations.  For starters, any family of solutions on a common time interval is uniformly equicontinuous.

\begin{lemma}\label{equicontinuous} \cite{filippov}
If $F(x)$ satisfies the basic conditions in a closed, bounded domain $D$, all solutions of the differential inclusion $\dot{x}\in F(x)$ are uniformly equicontinuous.
\end{lemma}

\begin{proof}
This result may be seen by considering the definition of a solution.  A solution $x(t)$ has a derivative $\dot{x}(t)\in F(x(t))$ almost everywhere, and it satisfies the Lebesgue integral equation 
$$x(t)=x(0)+\int_0^t \dot{x}(s)ds$$
Using the bound of $F$ in $D$ from lemma \ref{bounded}, we see that solutions are equicontinuous: 
$$|x(s_1)-x(s_2)|=|\int_{s_2}^{s_1} \dot{x}(s)ds|\leq \int_{s_2}^{s_1} |\dot{x}(s)|ds\leq \int_{s_2}^{s_1} M ds= M|s_1-s_2|$$
\end{proof}\par

Using the preceding lemmas and theorems, we can also show that any solution can be continued until it reaches the boundary of a compact domain.  The basic intuition of this claim is clear; if our solution terminates somewhere in the interior of a compact set, we can extend it using the existence theorem.  Below, we state and prove this result more rigorously.

\begin{theorem} \label{solcont} \cite{filippov}
Let the set-valued function $F(x)$ satisfy the basic conditions in a closed, bounded domain $D$.  Then each solution of the differential inclusion $\dot{x}\in F(x)$ lying within $D$ can be continued on both sides up to the boundary of the domain $D$.
\end{theorem} 

\begin{proof}
From the existence theorem \ref{exist}, we know that at any initial condition there is a solution $x(t)$ to the differential inclusion on some closed interval $[0,c_1]$.  The idea of this proof is to now consider the point $x(c_1)$ and extend the solution from there by again using the existence theorem.  This process is iterated indefinitely, and it either yields a solution defined for all time (meaning the solution remains in the interior of $D$ for all forwards time) or the position of the solution at the endpoints limits to the boundary of $D$.  The process in backwards time is symmetric.\par 
Now, more formally, take an arbitrary $x_0\in D$.  There is some $\epsilon_1$ such that $\overline{B_{\epsilon_1}(x_0)}$ is contained in the interior of $D$.  Following the method of the proof of the existence theorem \ref{exist}, there is a solution $x:[0,c_1]\to \overline{B_{\epsilon_1}(x_0)}$ where $c_1=\frac{\epsilon_1}{m}$ and $m=\sup_D|F(x)|$ ($m<\infty$ by lemma \ref{bounded}).  Denoting the boundary of $D$ by $\Gamma$, if $d(x(c_1),\Gamma)> \epsilon_1$ then we can extend the solution to a further interval of length $c_1$.  We either repeat this process indefinitely (giving us a solution which remains in $D$ for all forwards time) or until $d(x(k c_1),\Gamma)\leq \epsilon_1$ for some $k$.  In the latter case, let $t_1=k c_1$ and  $x_1=x(t_1)$.\par 
Choosing $\epsilon_2<d(x_1,\Gamma)<\epsilon_1$, we repeat this process, letting $t_2=t_1+ j c_2$ and $x_2=x(t_2)$ if we reach a point that $d(x(t_1+j c_2),\Gamma)\leq \epsilon_2$.  In fact, we may iterate this process either until the algorithm yields a solution remaining in the interior of $D$ for all time at some $i^{th}$ step or we get sequences 
$$t_1<t_2<\cdots, \hspace{1cm} x_1,x_2,\cdots$$\par 
If $t_i\to\infty$ as $i\to\infty$, then the solution $x$ remains in $D$ for all forwards time.  Otherwise, there exists some $T$ such that $t_i<T$ for all $i$.  Thus $\{t_i\}$ is a bounded, monotonic sequence, and hence limits to some $t^*$.  This bound implies that $\epsilon_i\to 0$ since $c_i=\frac{\epsilon_i}{m}$.  We also see that the $x_i$ converge to some $x^*$ because by the equicontinuity of solutions (lemma \ref{equicontinuous}), $|x(t_i)-x(t_j)|\leq m|t_i-t_j|$.  Letting $x(t^*)=x^*$, we obtain a solution $x:[0,t^*]\to D$ which reaches the boundary of $D$.  
\end{proof}

These theorems provide us with the initial structure necessary to begin some analysis of differential inclusions.  However, we still lack anything like a flow for these systems.  This omission is very unfortunate, because a lot of information can be gleaned from topological information on flows.  Without them (or something like them) we cannot have the concept of Omega limit sets, and something like Conley Index analysis is impossible.\par 
Thus, correcting this omission is highly desirable.  Since differential inclusions do not necessarily have unique solutions for a given initial condition, they cannot be expected to have a flow associated with them. But perhaps we can define an object that allows us to perform similar analysis, some single object that describes the entire collection of possible solutions to a differential inclusion at once. Richard McGehee has proposed such an object, which he calls a \textit{multiflow} \cite{mcgehee}.  

\section{Multiflows}$\,$\par 

\subsection{Defining Multiflows}

Before we can define multiflows we need some background information.\par 

\begin{definition} If $X$ is a set, then a \textbf{relation} on $X$ is any subset of $X\times X$.
\end{definition}
\begin{definition} The \textbf{composition of two relations} $F$ and $G$ on a set $X$ is the relation
$$F\circ G=\{(x,z)\in X\times X:\exists y\in X \text{ such that } (x,y)\in G,(y,z)\in F\}$$
\end{definition} 
We also need to introduce some notation.  Let $\Phi\subset[0,\infty)\times X\times X$; we write 
$$\Phi^t=\{(x,y):(t,x,y)\in\Phi\}$$
That is, for each $t\geq 0$, $\Phi^t$ defines a relation on $X$.  With these concepts in mind, we can now define multiflows.\\
\begin{definition} Let $X$ be a compact metric space.  A \textbf{multiflow} on $X$ is a closed subset of $[0,\infty)\times X\times X$ satisfying the two monoid properties:
\begin{enumerate}
\item $\Phi^0=\{(x,x)\in X\times X\}$
\item $\Phi^{t+s}=\Phi^t\circ\Phi^s$ for all $t,s\geq 0$.  
\end{enumerate}
\end{definition}
The concept of a multiflow arises from considering the graph of a flow and the closed graph theorem.  The closed graph theorem tells us that a function is continuous if and only if its graph is closed.  Then if we identify the flow $\varphi$ with its graph $\phi\subset\R\times X\times X$ 
we see that the definitions of a flow and a multiflow are almost identical.  Indeed, the restriction of any flow to forward time is automatically a multiflow.  There are only two differences between the objects, and both of these differences are motivated by necessities of differential inclusions.\par
The first difference is that a multiflow only considers forward time, making it more closely akin to a semiflow  than a complete flow.  The possibility of intersecting trajectories necessitates this difference.  With differential inclusions, it is possible for solutions that begin at distinct initial conditions to reach the same point in finite time.  Therefore a solution could move forwards for time $t$ to one location, then backwards for time $-t$ to a location other than the initial condition.  This makes retaining the group action of $\R$ impossible because the identity requirement would not hold in general. Therefore we only examine solutions in forward time and settle for a monoid action.  Note, however, that it is still quite straightforward to examine backwards time behaviour as a separate system.\par 
 \begin{figure}[h]
  \centering
  \subfloat{\includegraphics[width=0.4\textwidth]{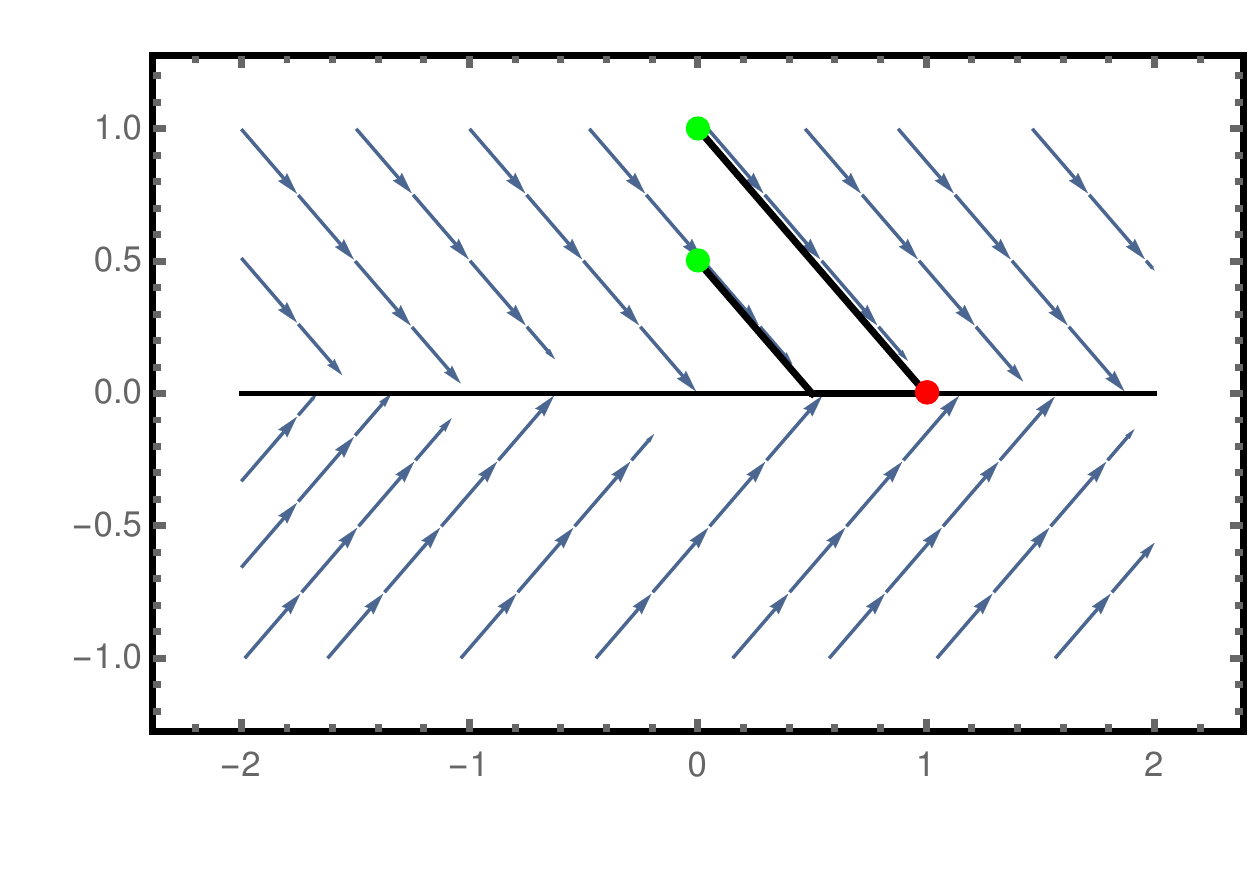}\label{fig:f1}}
  
  \caption{In this simple Filippov system, two different initial conditions reach the same point at time $t=1$.  This collision means that any analogue of a flow for Filippov systems cannot have a group action: $\varphi^1\circ\varphi^{-1}$ would not be the identity.}
\end{figure}
The other difference is that for a flow $\varphi$, each $x\in X$ has a unique $\varphi_t(x)$ such that $(x,\varphi_t(x))\in \phi_t$.  As we have discussed at length, such a condition cannot possibly hold when we consider differential inclusions.  Therefore the multiflow is the object which retains all of the structure of a flow that we cannot immediately rule out as applying to differential inclusions.\par 
Of course, we have also defined multiflows only over compact spaces $X$, and not all topological spaces.  This difference, however, is not fundamental, and is chiefly designed for analytic convenience.  This field of study is very new, raw, and abstract, so adding the assumption of compactness gives an extremely helpful simplification.  Once multiflows over compact spaces have been more thoroughly studied, if they turn out to be useful, further research could delve into what happens without this assumption.\par 

\subsection{Differential Inclusions as Multiflows}$\,$\par

We now present the main result of this paper: that any differential inclusion which satisfies the basic conditions gives rise to a multiflow.  Let $\dot{x}\in F(x)$ be any Filippov system defined on an open domain $G\subset\R^n$, and let $K$ be any non-empty, connected, compact subset of $G$. We are interested in considering all solutions of the differential inclusion which are contained entirely in $K$.  If a solution begins in $K$, but then leaves, we only monitor the solution up until the time it leaves.  We want to show that the union of the graphs of all of these solutions forms a multiflow.\par 
More explicitly, define the object $\Phi$ to be the set of all points $\{(T,a,b)\in \R^+\times K\times K\}$ such that there exists a solution $x(t):[0,T]\to K$ to the differential inclusion $\dot{x}\in F(X)$ with $x(0)=a$ and $x(T)=b$.  
\begin{theorem}
The set $\Phi$ is a multiflow over $K$.
\end{theorem}
\begin{proof}
The monoid properties are relatively trivial to see, although it does take a bit of space to write their proof. We see that $\Phi^0=\{(a,a)\in K\times K\}$ because by theorem \ref{exist}, for each $a\in K$, there is at least one solution, and obviously a solution cannot begin at $a$ and go to any other point in zero time.  \par 
Next, note that the second monoid property, $\Phi^{t+s}=\Phi^t\circ\Phi^s$ for all $t,s\geq 0$, is equivalent to the following statement: $(t+s,a,c)\in\Phi$ if and only if there exists $b\in K$ such that $(t,a,b)\in\Phi$ and $(s,b,c)\in\Phi$.  In our case, points in $\Phi$ may be written as $(t,x(0),x(t))$.  Then we must show two things.  First, if there is a solution $z$ such that $z(0)=a$ and $z(t+s)=c$, then there must be some point $b$ and solutions $x$ and $y$ such that $x(0)=a$, $x(t)=b=y(0)$, and $y(s)=c$.  Conversely, if there is some point $b$ and solutions $x$ and $y$ such that $x(0)=a$, $x(t)=b=y(0)$, and $y(s)=c$, then there must be a solution $z$ such that $z(0)=a$ and $z(t+s)=c$.\par 
Let us first assume that there is some point $b$ and solutions $x$ and $y$ such that $x(0)=a$, $x(t)=b=y(0)$, and $y(s)=c$.  In brief, pasting these solutions together yields the desired solution.  More rigorously, define the function $z:[0,t+s]\to K$ by the equation
$$z(r) = 
\begin{cases}
x(r) & r\leq t\\
y(r-t) & r \geq t
\end{cases}$$
It is clear that $z$ is absolutely continuous since both $x$ and $y$ are, and by its definition it is obvious that $z(0)=a$ and $z(t+s)=c$.  It also satisfies the differential equation $\dot{z}\in F(z)$ almost everywhere because for almost all $r\in[0,t]$,
$$\frac{d}{dr}z(r)=\frac{d}{dr}x(r)\in F(x(r))=F(z(r))$$
and for almost all $r\in[t,t+z]$
$$\frac{d}{dr}z(r)=\frac{d}{dr}y(r-t)\in F(y(r-t))=F(z(r))$$
Thus $z$ is a solution to the differential inclusion.\par 
Now assume that there is a solution $z$ such that $z(0)=a$ and $z(t+s)=c$. In brief, splitting this function into two functions at time $t$ yields the desired solutions.  More rigorously, define the functions $x:[0,t]\to K$ and $y:[0,s]\to K$ by the equations
$$ x(r) = z(r)\hspace{1cm} y(r) = z(r+t)$$
Again, it is clear that $x$ and $y$ are absolutely continuous functions which evaluate to the desired points at the appropriate times.  Additionally, 
$$\frac{d}{dr}x(r)=\frac{d}{dr}z(r)\in F(z(r))=F(x(r))$$
almost everywhere and 
$$\frac{d}{dr}y(r)=\frac{d}{dr}z(r+t)\in F(z(r+t))=F(y(r))$$
almost everywhere.  Therefore $x$ and $y$ are the desired solutions to the differential inclusion, and $\Phi$ satisfies the monoid properties.\par 
The difficultly of this proof comes from showing that $\Phi$ is closed.  Luckily, Filippov's theorems do much of the hard work for us.  Let $(T,a,b)$ be a limit point of $\Phi$; we will show that  $(T,a,b)\in\Phi$.\par 
Since $(T,a,b)$ is a limit point of $\Phi$, there is some sequence of points in $\Phi$
$$(T_n,x_n(0),x_n(T_n))\to(T,a,b)$$
where each $x_n(t)\in F(x_n(t))$ for almost all $t$ in the interval $[0,T_n]$.\par
The basic idea of the proof is to find a subsequence of $\{x_n\}$ which exist on (or can be extended to) the common interval $[0,T]$.  Then we apply the Arzela-Ascoli theorem to this family of solutions in order to get a uniformly convergent subsequence.  By theorem \ref{unisol}, this subsequence converges to a solution $x^*(t)$; the proof will then be complete once we show that $x^*(0)=a$, $x^*(T)=b$, and $x^*(t)\in K$ $\forall t\in[0,T]$.\par 
We begin by taking a compact neighbourhood $N$ of $K$; that is, $N$ is a compact set satisfying\footnote{We use the notation $N^0$ to denote the interior of $N$} $$K\subset N^0\subset N\subset G$$\par 
By lemma \ref{bounded} there is some constant $M$ such that $|F(x)|\leq M$ on $N$. Combining that result with lemma \ref{equicontinuous}, any family of solutions $\{x:[0,T]\to N\}$ is equicontinuous and 
$$|x(s_1)-x(s_2)|\leq M|s_1-s_2|$$\par 
At this point, however, the sequence of solutions we are considering are not necessarily defined on the common interval of time $[0,T]$, and in order to get equicontinuity and apply Arzela-Ascoli, we need them to be.  If $T_n>T$ then this presents no obstacle, as we simply consider $x_n|_{[0,T]}$.  However, we must show that for sufficiently large $n$ we can extend $x_n$ to the interval $[0,T]$ even if $T_n<T$.\par 
By theorem \ref{solcont}, we know that any solution can be extended at least until it reaches the boundary of $N$.  Since $K\subset N^0$, we can extend any $x_n$ to be defined on some interval $[0,T_n']$, where $T_n'>T_n$.  Let $\delta:=d(K,\overline{N})>0$, and choose an $n_0$ such that $n\geq n_0$ implies that $|T_n-T|<\frac{\delta}{M}$.  Then for such $n$, if $T_n<T_n'<T$ we get that
$$|x_n(T_n)-x_n(T_n')|<M|T_n-T_n'|<\delta$$
and so $x_n(T_n')\in N^0$, and hence may be continued further.  Thus we may assume that for $n\geq n_0$, the solution $x_n$ may be extended to the interval $[0,T]$.\par 
We are now in a position to apply Arzela-Ascoli; the solutions are clearly uniformly bounded (they are all contained in the compact set $N$) and we know that they are equicontinuous.  Thus, there is a convergent subsequence of $\{x_n:[0,T]\to N\}$.  By theorem \ref{unisol}, this subsequence converges to a solution of the differential inclusion; let us call this solution $x^*(t)$.  Since $x_n(0)\to a$ by definition, it is clear that $x^*(0)=a$.  Then to show that $(T,a,b)\in\Phi$, we just need to show that $x^*(T)=b$ and that $x^*(t)\in K$ $\forall t\in[0,T]$.\par 
To show that $x^*(T)=b$, we will show that for any $\varepsilon$, $|x_n(T)-b|<\varepsilon$ for sufficiently large $n$.
\begin{align*}
|x_n(T)-b|&=|x_n(T)-x_n(T_n)+x_n(T_n)-b|\\
&\leq |x_n(T)-x_n(T_n)|+|x_n(T_n)-b|\\
&\leq M|T-T_n|+|x_n(T_n)-b|
\end{align*}
Since $T_n\to T$ and $x_n(T_n)\to b$ by assumption, we can guarantee that $|T-T_n|<\frac{\varepsilon}{2M}$ and $|x_n(T_n)-b|<\varepsilon/2$ for sufficiently large $n$, and so $x^*(T)=b$.\par 
Now, for sake of contradiction, assume that $x^*(t)\in N\setminus K$ for some $t\in[0,T]$.  
Let $\tau:= T-t$.  For sufficiently large $n$, $|T-T_n| < \tau$.  Then for these large $n$, $x_n(t)\in K$, since we have chosen $n$ large enough to guarantee that $t \in [0,T_n]$ and $x_n:[0,T_n]\to K$ by definition.  Then $x_n(t)$ must limit to a point in $K$ as $n \to\infty$ since $K$ is compact.  Thus we have a contradiction, and we see that $x^*:[0,T]\to K$.\par 
Thus $(T,x^*(0),x^*(T))=(T,a,b)\in\Phi$, and so $\Phi$ is a multiflow.  
\end{proof}

\subsection{Concepts Related to Multiflows}$\,$\par

Several other researchers have attempted to generalize the concept of a flow to differential inclusions, and their goals are often similar to the goals of multiflows.  The oldest attempt we can find in this direction came from Roxin \cite{roxin}, who developed a generalized dynamical system in order to study general control systems.  Later, Ball \cite{ball} defined a generalized semiflow, and Melnik and Valero \cite{melnik} then defined a multi-valued semiflow.  Recently, Oyama described a set-valued dynamical system \cite{oyama}.  Each definition is distinct in general, but they share many thematic similarities.  We will briefly examine the definition given byOyama and compare it to the concept of multiflows.\par
The obeject described in\cite{oyama} is very similar to multiflows.  Oyama calls the set-valued map $\Phi:[0,\infty)\times X\to X$ on a compact subset $X\subset\R^n$ a  \textbf{\textit{set-valued dynamical system}} if it meets the following conditions:
\begin{enumerate}
\item $\Phi_t(x)$ is nonempty for all $t,x$
\item $\Phi_0(x)=x$
\item $\Phi_t(\Phi_s(x))=\Phi_{t+s}(x)$
\item $\Phi$ is compact valued and upper-semicontinouous.  
\end{enumerate} \par 
Conditions $(2)$ and $(3)$ are equivalent to the monoid conditions of multiflows.  Here, condition $(4)$ is actually equivalent to the closure condition of multiflows since the space $X$ is compact.  But condition $(1)$ is much stronger than in multiflows, and it requires that solutions of the differential inclusion remain in a compact subset for all time.  As we have noted several times in this paper, that condition will not generally be met by differential inclusions, or even the prominent example of Filippov systems.  In fact, the largest difficulty in proving that differential inclusions give rise to multiflows was accounting for the solutions not all continuing for all time.  This distinction means that multiflows can describe a much larger class of differential inclusions than existing systems.  \par 
The fact that multiflows allow us to study dynamical systems without worrying about whether or not solutions exist for all time is one of its most important features, and is made possible by the shift in perspective from maps to closed sets.  Any time $t$ relation $\Phi^t$ is allowed to be the empty set, and so we do not need to make the demands that set-valued maps do.  This feature also helps deal with the complication of finite time blowup from ordinary ODEs (think of the simple ODE $\dot{x}=x^2$ and finite time blowup), and so we avoid any complications like local flows.  In this way, multiflows can be used to describe a very broad class of dynamical systems.\par 
As it stands now, multiflows do not have many practical applications.  However, we hope to take many of the ideas present in existing frameworks and adapt them to fit multiflows.  In this way we hope to develop a system that can help us understand a wide range of potentially non-unique dynamical systems, including, of course, Filippov systems.\par

\section{Conclusions and Future Work} $\,$ \par
Filippov systems, and differenial inclusions in general, are becoming more and more popular in scientific modelling.    Although these systems have many undersirable features--a lack of uniqueness chief among them--they are deeply immeshed in the scientific community.  Since these models seem to be here to stay, it would be nice to have a robust framework to analyze them with.\par 
We have seen throughout this paper that differential inclusions cannot give rise to flows in general, which removes a valuable tool from our mathematical arsenal.  However, we have also seen that multiflows retain as many of the features of flows as possible, and that basic differential inclusions give rise to multiflows.  What remains to be seen is whether or not framing differenial inclusions as multiflows actually provides any useful information.  Hopefully, multiflows will allow us to define generalizations of concepts like $\omega$-limit sets, chain recurrence, and Conley Index theory that are suitable for these systems.  One immediate question that needs to be answered is whether isolating neighborhoods and attractors can be robustly defined for multiflows; if so, then multiflows will likely be a very useful tool.
\par  
Many open questions remain surrounding differential inclusions.  Their behaviour can be extremely bizzarre.  However, Filippov showed that these systems do have some familiar properties. Solutions exist.  Solutions are bounded and equicontinuous in compact domains, and they continue until they reach the edge of the  domain.  And the uniform limit of a sequence of solutions is a solution.  With all of these theorems we can see that individual solutions to Filippov systems behave a lot like solutions to typical differential equations.  Examining the entire set of solutions to a given Filippov system is more difficult, and understanding their behavior under perturbation is even harder.  But in the future, we hope that multiflows will give us more tools to analyze these complicated dynamical systems.



\newpage

\begin{thebibliography}{9}

\bibitem{bernardo} 
M. Di Bernardo, C.J. Budd, A.R. Champneys \& P. Kowalczyk.
\textit{Piecewise-smooth
Dynamical Systems Theory and Applications}. 
Springer, 2008.

\bibitem{ball}
J.M. Ball.
\textit{Continuity properties and global attractors of generalized semiflows and the Navier–
Stokes equations}. J. Nonlinear Sci. 7(5): 475–502, 1997.

\bibitem{carabllo}
T. Caraballo, P. Marin-Rubio \& J.C. Robinson
\textit{A Comparison between Two Theories for
Multi-Valued Semiflows and Their Asymptotic
Behaviour}.
Set-Valued Analysis 11: 297–322, 2003.


\bibitem{filippov} 
A.F. Filippov. 
\textit{Differential Equations with Discontinuous Righthand Sides}. 
Dept. Math., Moscow State University, U.S.S.R., Kluwer Acad. Pub., Boston, MA, 1988.
 
\bibitem{guckenheimer}
J. Guckenheimer. 
\textit{Piecewise-smooth
Dynamical Systems Theory and Applications}.
SIAM Review, 50:606-609, 2008.
 
\bibitem{hill}
K. Hill.
\textit{Bifurcation Analysis of a Piecewise-Smooth Arctic Energy Balance Model}.
Thesis, Northwestern University, 2017.

\bibitem{jeffrey}
M.R. Jeffrey.
\textit{Hidden Dynamics in Models of Discontinuity and Switching}.
Physica D, 273:34-45, 2014.
 
\bibitem{kuznetsov}
Yu. A. Kuznetsov , S. Rinaldi \& A. Gragnani.
\textit{One-Parameter Bifurcations in Planar Filippov Systems} 
International Journal of Bifurcation and Chaos, 13(08):2157-2188, 2003.

\bibitem{leifeld}
J. Leifeld.Ball, J. M.: Continuity properties and global attractors of generalized semiflows and the Navier–
Stokes equations, J. Nonlinear Sci. 7(5) (1997), 475–502.
\textit{Smooth and Nonsmooth Bifurcations in Welander’s Convection Model}.
Thesis, University of Minnesota, 2016.

\bibitem{leifeldcode}
J. Leifeld.  Latex Code.  Personal Communication, 2018.

\bibitem{mcgehee}
R. McGehee.  Personal Communication, 2017-2018.

\bibitem{melnik}
V. Melnik \& J. Valero.  
\textit{On Attractors of Multivalued Semi-Flows and
Differential Inclusions}.
Set-Valued Analysis 6: 83–111, 1998.

\bibitem{oyama}
D. Oyama.  
\textit{Lecture Notes on Set-Valued Dynamical Systems}.  Lecture Notes, University of Tokyo, https://www.u-tokyo.ac.jp, 2014.

\bibitem{roxin}
E. Roxin.
\textit{On Generalized Dynamical Systems Defined by
Contingent Equations}.
Journal of Differential Equations, 1: 188-205, 1965.


\bibitem{welander}
P. Welander.
\textit{A Simple Heat-Salt Oscillator}.
Dynamics of Atmospheres and Oceans, 6(4):233-242, 1982.


 


\end{thebibliography}
\end{document}